\documentclass[11pt]{article}

\usepackage{amsmath}
\usepackage{amsthm}
\usepackage{amsfonts}
\usepackage{amssymb}
\usepackage{latexsym}
\usepackage{color}
\usepackage[noautoscale]{youngtab}

\usepackage{fullpage}

\usepackage{color, colortbl}
\usepackage{array}
\usepackage{multirow}

\usepackage{tikz}

\numberwithin{equation}{section}

\newtheorem{theorem}{Theorem}[section]

\newtheorem{lemma}[theorem]{Lemma}
\newtheorem{remark}[theorem]{Remark}

{
\theoremstyle{definition}
\newtheorem{definition}[theorem]{Definition}

}

\usepackage[colorlinks=true,
linkcolor=webgreen,
filecolor=webblue,
citecolor=webgreen]{hyperref}

\definecolor{webgreen}{rgb}{0,.5,0}
\definecolor{webblue}{rgb}{.6,0,0}

\usepackage[symbol]{footmisc}

\begin{document}

\title{On 1-11-representability and multi-1-11-representability of graphs}

\author{Mohammed Alshammari\footnote{Department of Mathematics and Statistics, University of Strathclyde, 26 Richmond Street, Glasgow G1 1XH, United Kingdom. Email: mohammed.s.alshammari@strath.ac.uk}, Sergey Kitaev\footnote{Department of Mathematics and Statistics, University of Strathclyde, 26 Richmond Street, Glasgow G1 1XH, United Kingdom. Email: sergey.kitaev@strath.ac.uk}, Chaoliang Tang\footnote{Shanghai Center for Mathematical Sciences, Fudan University, 220 Handan Road, Shanghai 200433, China. Email: cltang22@m.fudan.edu.cn}, Tianyi Tao\footnote{Shanghai Center for Mathematical Sciences, Fudan University, 220 Handan Road, Shanghai 200433, China. Email: tytao20@fudan.edu.cn} and Junchi Zhang\footnote{Shanghai Center for Mathematical Sciences, Fudan University, 220 Handan Road, Shanghai 200433, China. Email: jczhang24@m.fudan.edu.cn}}

\date{\today}

\maketitle

\noindent\textbf{Abstract.}   Jeff Remmel introduced the concept of a $k$-11-representable graph in 2017. This concept was first explored by Cheon et al.\ in 2019, who considered it as a natural extension of word-representable graphs, which are exactly 0-11-representable graphs. A graph $G$ is $k$-11-representable if it can be represented by a word $w$ such that for any edge (resp., non-edge) $xy$ in $G$ the subsequence of $w$ formed by $x$ and $y$ contains at most $k$ (resp., at least $k+1$) pairs of consecutive equal letters. A remarkable result of Cheon at al.\ is that {\em any} graph is 2-11-representable, while it is still unknown whether every graph is 1-11-representable. Cheon et al.\ showed that the class of 1-11-representable graphs is strictly larger than that of word-representable graphs, and they introduced a useful toolbox to study 1-11-representable graphs, which was extended by additional powerful tools suggested by Futorny et al.\ in 2024.

In this paper, we prove that all graphs on at most 8 vertices are 1-11-representable hence extending the known fact that all graphs on at most 7 vertices are 1-11-representable.  Also, we discuss applications of our main result in the study of multi-1-11-representation of graphs we introduce in this paper analogously to the notion of multi-word-representation of graphs suggested by Kenkireth and Malhotra in 2023.   
 \\

\noindent {\bf AMS Classification 2010:}  05C62


\section{Introduction}

Various methods for representing graphs extend far beyond the conventional use of adjacency or incidence matrices; for example, see \cite{Spinrad} for a discussion. Of particular relevance to our paper are representations of graphs by words or sequences, where adjacency between a pair of vertices is determined by the occurrence of a fixed pattern in the subword or subsequence formed by the vertices. For instance, in the extensively studied word-representable graphs \cite{KL15}, first studied in \cite{KP08} and defined in Section~\ref{wrg-sec}, edges are determined by the strict alternation of vertices. 

Another representation method is $k$-11-representation, introduced by Jeff Remmel in 2017 and defined in Section~\ref{k-11-repr-sub}, where at most $k$ violations of strict alternation are allowed to define an edge between two vertices. Consequently, word-representable graphs correspond precisely to 0-11-representable graphs. The concept of $k$-11-representable graphs was first studied by Cheon et al. in 2019 \cite{CKKKP2019}.

While not all graphs are word-representable, all graphs are known to be 2-11-representable  \cite{CKKKP2019}. The most intriguing open question in the theory of $k$-11-representable graphs is whether all graphs are 1-11-representable, and it remains challenging to predict an answer to this question. Recent research in this area has focused on developing powerful tools to study 1-11-representable graphs \cite{FKP}; see Section~\ref{k-11-repr-sub}.

\subsection{Our main results and the organization of the paper}

In Section~\ref{prelim-sec}, we introduce the necessary definitions and known results that will be used throughout this paper. In Section~\ref{graph-8-vertices}, we prove that all graphs with at most 8 vertices are 1-11-representable, thereby extending the previously known result that all graphs with at most 7 vertices are 1-11-representable \cite{CKKKP2019, FKP}. In Section~\ref{multi-1-11-repr-sec}, we introduce the concept of the multi-$k$-11-representation number of a graph, which generalizes the notion of the multi-word-representation number of a graph \cite{KM}.

As an application of our main results in this paper, in Section~\ref{multi-1-11-repr-sec}, we demonstrate that all graphs with at most 24 vertices have a multi-1-11-representation number of at most 2. Finally, in Section~\ref{concl-sec}, we provide concluding remarks and outline open problems.

\section{Preliminaries}\label{prelim-sec}

An orientation of a graph is {\em transitive}, if the presence of the edges $u\rightarrow v$ and $v\rightarrow z$ implies the presence of the edge $u\rightarrow z$. An undirected graph $G$ is a {\em comparability graph} if $G$ admits a transitive orientation. 

\subsection{Word-representable graphs and semi-transitive orientations}\label{wrg-sec}

\noindent 
Two letters $x$ and $y$ alternate in a word $w$ if, after deleting all letters in $w$ except for $x$ and $y$, we obtain either the word $xyxy\cdots$ or $yxyx\cdots$ (of even or odd length).  A graph $G=(V,E)$ is {\em word-representable} if and only if there exists a word $w$
over the alphabet $V$ such that letters $x$ and $y$, with $x\neq y$, alternate in $w$ if and only if $xy\in E$. The word $w$ is called a {\em word-representant} for~$G$. 

The unique minimum (by the number of vertices) non-word-representable graph with 6 vertices is the wheel graph $W_5$. Moreover,  there are 25 non-word-representable graphs on 7 vertices. Notably, the original list of 25 non-word-representable graphs with 7 vertices, as presented in~\cite{KL15}, contains two incorrect graphs. For the corrected catalog of these 25 graphs, we refer the reader to~\cite{KitaevSun}.

%

An orientation of a graph is {\em semi-transitive} if it is acyclic, and for any directed path $v_0\rightarrow v_1\rightarrow \cdots \rightarrow v_k$ either there is no edge from $v_0$ to $v_k$, or $v_i\rightarrow v_j$ is an edge for all $0\leq i<j\leq k$. An induced subgraph on at least four vertices $\{v_0,v_1,\ldots,v_k\}$ of an oriented graph is a {\em shortcut} if it is acyclic, non-transitive, and contains both the directed path $v_0\rightarrow v_1\rightarrow \cdots \rightarrow v_k$ and the edge $v_0\rightarrow v_k$, that is called the {\em shortcutting edge}. A semi-transitive orientation can then be alternatively defined as an acyclic shortcut-free orientation. A fundamental result in the area of word-representable graphs is the following theorem.

\begin{theorem}[\cite{Halldorsson}]\label{semi-trans-thm} A graph is word-representable if and only if it admits a semi-transitive orientation. \end{theorem} 

For instance, it follows from Theorem~\ref{semi-trans-thm} that every 3-colourable graph is word-representable (simply direct each edge from a lower colour to a higher one).
In the literature, word-representable graphs are often referred to as semi-transitive graphs.

\subsection{$k$-11-representable graphs}\label{k-11-repr-sub}

\noindent
A {\em factor} in a word $w_1w_2\ldots w_n$ is a word $w_iw_{i+1}\ldots w_j$ for $1\leq i\leq j\leq n$.
 For any word $w$, let $\pi(w)$ be the {\em initial permutation} of $w$ obtained by reading $w$ from left to right and recording the leftmost occurrences of the letters in $w$. Denote by $r(w)$ the {\em reverse} of $w$, that is, $w$ written in the reverse order. Finally, for a pair of letters $x$ and $y$ in a word $w$, let $w|_{\{x,y\}}$ be the subword induced by the letters $x$ and~$y$.  For example, if $w=42535214421$ then $\pi(w)=42531,\ r(w)=12441253524,$ and $w|_{\{4,5\}}=45544$.

Let $k\geq 0$. A graph $G=(V,E)$ is {\em $k$-$11$-representable} if there exists a word $w$ over the alphabet $V$ such that the word $w|_{\{x,y\}}$  contains in total at most $k$ occurrences of the factors in $\{xx,yy\}$ if and only if $xy$ is an edge in $E$. Such a word $w$ is called $G$'s {\em $k$-$11$-representant}. 
Note that $0$-$11$-representable graphs are precisely word-representable graphs, and that $0$-$11$-representants are precisely word-representants.
A graph $G=(V,E)$ is {\em permutationally $k$-$11$-representable} if it has a $k$-$11$-representant that is a concatenation of permutations of $V$.
 The ``11'' in ``$k$-$11$-representable'' refers to counting occurrences of the {\em consecutive pattern} 11 in the word induced by a pair of letters $\{x,y\}$, which is exactly the total number of occurrences of the factors in $\{xx,yy\}$.  

A {\em uniform} (resp., {\em $t$-uniform}) representant of a graph $G$ is a word, satisfying the required properties, in which each letter occurs the same (resp., $t$) number of times. It is known that each word-representable graph has a uniform representant~\cite{KP08}, 
the class of 2-uniform representable graphs is exactly the class of circle graphs~\cite{KL15}, while the class of $2$-uniform $1$-$11$-representable graphs is the 
class of interval graphs~\cite{CKKKP2019}.  The main result in~\cite{CKKKP2019} is the following theorem.

\begin{theorem}[\cite{CKKKP2019}]\label{begin-end} Every graph $G$ is permutationally $2$-$11$-representable.
\end{theorem}

Thus, when determining whether each graph is $k$-$11$-representable for a fixed $k$, the only case left to study is $k=1$ (as the answer is no for $k=0$ and yes for $k\geq 2$).

\subsection{Known tools to study 1-11-representable graphs}\label{known-tools-sec}

Each word-representable graph  is 1-11-representable. Indeed, if $w$ is a word-representant of $G$ then, for instance,  $ww$ or  $r(\pi(w))w$ are its 1-11-representants. Moreover, each graph on at most 7 vertices is 1-11-representable \cite{CKKKP2019,FKP}. The key tools to study 1-11-representation of graphs from~\cite{CKKKP2019,FKP} can be unified as follows.

\begin{lemma}[\cite{CKKKP2019}]\label{type1-lem}  \ 

{\rm (a)} Let $G_1$ and $G_2$ be $1$-$11$-representable graphs. Then their disjoint union, glueing them in a vertex or connecting them by an edge results in a $1$-$11$-representable graph.

{\rm (b)} If $G$ is $1$-$11$-representable then for any edge $xy$ adding a new vertex adjacent to $x$ and $y$ only, gives a $1$-$11$-representable graph.
\end{lemma}

\begin{lemma}[\cite{CKKKP2019}]\label{type2-lem} Let $G$ be  a word-representable graph, $A\subseteq V$ and $v\in V$. Then

{\rm (a)} $G\setminus \{xy\in E(G)\ |\ x,y\in A\}$ is a $1$-$11$-representable graph;

{\rm (b)} $G\setminus \{uv\in E(G)\ |\ u\in N_A(v)\}$ is a $1$-$11$-representable graph.
\end{lemma}

In particular, Lemma~\ref{type2-lem}(b) is frequently referred to in this paper as the ``star operation'' or ``adding a star'', and it is used as follows: to prove the 1-11-representability of a given graph, we identify a set of new edges, all sharing the same vertex as an endpoint, and demonstrate that the resulting graph is word-representable.

\begin{lemma}[\cite{CKKKP2019}]\label{type3-lem} Let $G$ be a graph with a vertex $v\in V$.  $G$ is $1$-$11$-representable if at least one of the following conditions holds:

{\rm (a)} $G\setminus v$ is a comparability graph;

{\rm (b)} $G\setminus v$ is a circle graph.
\end{lemma}

\begin{lemma}[\cite{FKP}] \label{new-tool-thm} Let $V_1,\ldots,V_k$ be pairwise disjoint subsets of $V$, the set of vertices of a word-representable graph $G$. We denote by $E(V_i)$ the set of all edges of $G$ having both end-points in $V_i$.  Then, the graph $H=G\backslash(\cup_{1\leq i\leq k} E(V_i))$, obtained by removing all edges belonging to $E(V_i)$ for all $1\leq i\leq k$, is $1$-$11$-representable. \end{lemma}

As a corollary of Lemma~\ref{new-tool-thm}, we obtain the following lemma, which is frequently used in this paper and referred to as ``adding a matching'' or ``applying a matching operation''.

\begin{lemma}[\cite{FKP}]\label{adding-matching} Let the graph $G$ be obtained from a graph $H$ by adding a matching
(that is, by adding new edges no pair of which shares a vertex). If $G$ is word-representable then $H$ is $1$-$11$-representable. \end{lemma}

\begin{lemma}[\cite{FKP}] Suppose that the vertices of a graph $G$ can be  partitioned into a comparability graph formed by vertices in $A=\{a_1,\ldots,a_k\}$ and an independent set formed by vertices in $B=\{b_1,\ldots,b_{\ell}\}$. Then $G$ is permutationally $1$-$11$-representable.  \end{lemma}

\section{Graphs on at most 8 vertices}\label{graph-8-vertices}

In what follows, $\chi(G)$ denotes the chromatic number of $G$. We say that a graph is {\em $(a_1,a_2,\ldots,a_k)$-colourable} if it can be coloured with $k$ colours, but not with $k-1$ colours, and the $i$-th colour class, corresponding to colour $i$, is the set $V_i=\{v_i,v'_i,v''_i,\ldots\}$ of size $a_i$. Our typical assumption, w.l.o.g., is that $a_1\geq a_2\geq \cdots\geq a_k$. However, in certain cases, we deviate from this assumption to be able to facilitate our arguments. 

\begin{remark}\label{clique-rem} Clearly, the vertices corresponding to $a_i=1$ must form a clique. \end{remark}

\begin{definition}
A {\em $(b_1b_2\ldots b_m)$-shortcut} is a shortcut with the directed path  $w_{b_1}\rightarrow w_{b_2}\rightarrow\cdots\rightarrow w_{b_m}$, where $w_{b_i}\in V_i$ for $1\leq i\leq m$. A $(b_1b_2 \ldots b_s$--)-shortcut is any $(c_1c_2 \ldots c_t)$-shortcut such that $b_i=c_i,$ for $i\in\{1,2\ldots,s\}$, and $t\geq s$.
\end{definition}

For sets of vertices $X$ and $Y$ in a graph, let $e(X,Y)$ denote the number of (directed or undirected) edges between $X$ and $Y$. For brevity, a singleton set $\{x\}$ is denoted as $x$. Additionally, for a graph $G$ with disjoint subsets of vertices $V_1,\ldots,V_m$, where $V_i$ is an independent set for $1\leq i\leq m$, let $G[V_1,\ldots,V_m]$ represent the induced $m$-partite subgraph of $G$ on the vertices in $\cup_{1\leq i\leq m}V_i$.  Finally, a {\em split graph} is a graph whose vertex set can be partitioned into a comparability graph and an independent set.

\begin{lemma}\label{high-chromatic-number} 
If a graph $G$ is $(k,1,1,\ldots,1)$-colourable, then $G$ is $1$-$11$-representable.
\end{lemma}

\begin{proof}
    Clearly, $G$ is a split graph with an independent set of size $k$, and by Theorem 6 in \cite{FKP}, any split graph is 1-11-representable.
\end{proof}

\begin{lemma}\label{symmetricvertices}
For an $(a_1,a_2,\ldots,a_k)$-colourable graph $G$, where $a_1\geq a_2\geq \cdots \geq a_i>a_{i+1}=a_{i+2}=\cdots =a_k=1$, if $e(V_s,V_t)=1$ for some $s\leq i<t$, then the vertex in $V_t$ and its unique neighbour in $V_s$ are adjacent to all vertices in $V_j$ for $j> i$.
\end{lemma}

\begin{proof}
 Assume that $V_t=\{v_t\}$ is adjacent to a vertex $v_s\in V_s$. Since $a_t=1$, by Remark~\ref{clique-rem}, the claim holds for $v_t$. Now, suppose that $v_s$ is not adjacent to some $v_{\ell}\in V_{\ell}$ for $\ell>i$. By recolouring the vertices in $V_s\backslash\{v_s\}$ in colour $t$ and the vertex $v_s$ in colour  $\ell$, we obtain a $(k-1)$-colouring of $G$, which  contradicts the assumption that $\chi(G)=k$. Therefore, $v_s$ must be adjacent to  all vertices in $V_j$ for $j> i$.
\end{proof}

The proof of the following theorem can be reduced to considering the 929 non-word-representable graphs on 8 vertices since any word-representable graph is 1-11-representable. Our final Section~\ref{concl-sec} contains an intriguing question about this. However, our arguments are not restricted to the 929 graphs -- we consider all graphs on 8 vertices based on their chromatic number and prove their 1-11-representability.    

\begin{theorem}\label{thm-main} All graphs on at most $8$ vertices are $1$-$11$-representable. \end{theorem}

\begin{proof} We begin with the easier cases and continue with the more involved ones. \\[-3mm]

\noindent
{\bf Case 1: $\chi(G)\leq 3$.} Any such graph is word-representable and hence 1-11-representable. \\[-3mm]

\noindent
{\bf Case 2: $\chi(G)=8$.} This is a complete graph which is word-representable and hence 1-11-representable. \\[-3mm]

\noindent
{\bf Case 3: $\chi(G)=7$.} By Lemma~\ref{high-chromatic-number}, $G$ is word-representable, and hence 1-11-representable. \\[-3mm]

Our strategy for the remainder of the proof is to consider a suitable $(a_1,a_2,\ldots, a_k)$-colouring of $G$. We then orient edges $uv$, where $u\in V_i$ and $v\in V_j$ with $i<j$, as $u\rightarrow v$. Next, we apply a star or a matching operation to add additional edges, oriented again from smaller colour to larger colours, to eliminate potential shortcuts. This process ensures a semi-transitive orientation, demonstrating that the resulting graph is word-representable and, consequently, that the original graph is 1-11-representable. \\[-3mm]

\noindent 
{\bf Case 4: $\chi(G)=6$.}
Then $G$ is either $(3,1,1,1,1,1)$-colourable or $(2,2,1,1,1,1)$-colourable. By Lemma~\ref{high-chromatic-number}, we can assume that $G$ is $(2,2,1,1,1,1)$-colourable, with its vertices coloured as shown in the picture below. No  edges are drawn in the picture (as is the case with all the pictures below), and, in particular, the vertices in $C=\{v_3,v_4,v_5,v_6\}$ form a clique. In the picture, colours $1$ to $6$ correspond to red, blue, green, orange, yellow, and black, respectively.

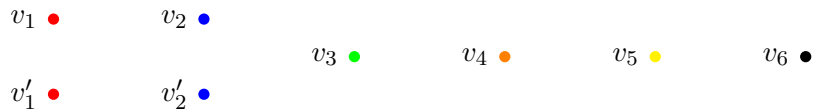
\begin{center}
    \begin{tikzpicture}

    \node[circle, fill=red, inner sep=1.5pt, label=left:$v_1$] (A) at (0,0) {};
    \node[circle, fill=red, inner sep=1.5pt, label=left:$v'_1$] (B) at (0,-1) {};
    \node[circle, fill=blue, inner sep=1.5pt, label=left:$v_2$] (C) at (2,0) {};
    \node[circle, fill=blue, inner sep=1.5pt, label=left:$v'_2$] (D) at (2,-1) {};
    \node[circle, fill=green, inner sep=1.5pt, label=left:$v_3$] (E) at (4,-0.5) {};
    \node[circle, fill=orange, inner sep=1.5pt, label=left:$v_4$] (F) at (6,-0.5) {};
    \node[circle, fill=yellow, inner sep=1.5pt, label=left:$v_5$] (G) at (8,-0.5) {};
    \node[circle, fill=black, inner sep=1.5pt, label=left:$v_6$] (H) at (10,-0.5) {};
    \end{tikzpicture}
\end{center}

Suppose that $v_1$ is not adjacent to a vertex in $C$. W.l.o.g., assume that $v_1v_3$ is not an edge. Clearly, $v'_1v_3$ is an edge; otherwise, $v_3$ can be coloured red, contradicting $\chi(G)=6$. But then, by Lemma~\ref{symmetricvertices}, $v'_1$ is adjacent to every vertex in $C$. 

The considerations above can also be applied to $v_2$ and $v'_2$ instead of $v_1$ and $v'_1$. By renaming $v_1$  (resp., $v_2$) and $v'_1$ (resp., $v'_2$), if necessary, we can assume that both $v_1$ and $v_2$ are adjacent to all vertices in $C$. Add any missing edges between $v'_1$ and the vertices in $C$ to obtain a graph $G'$, and rename the vertices in $C$, if necessary, so that the neighbours of $v'_2$ in $C$ are in the set $C'=\{v_i,v_{i+1},\ldots,v_6\}$ for $3\leq i\leq 7$ (note that $C'$ may be empty). Finally, orient the edges in $G'$  as $v_i\rightarrow v_j$ and $v'_i\rightarrow v_j$,  for $1\leq i<j\leq 6$, and $v_1\rightarrow v'_2$ and $v'_1\rightarrow v'_2$ (if any of these edges exists). It is easy to check that the obtained orientation is semi-transitive (in fact, transitive), so by Theorem~\ref{semi-trans-thm}, $G'$ is word-representable, and by Lemma~\ref{type2-lem}(b), $G$ is 1-11-representable. \\[-2mm]

\noindent
{\bf Case 5: $\chi(G)=5$.} The only possible shortcuts in this graph are (12345)-, (1234)-, (1235)-, (1245)-, (1345)-, or (2345)-shortcuts and possible missing edges appear only in $G[V_1,V_3]$, $G[V_1,V_4]$, $G[V_2,V_4]$, $G[V_2,V_5]$, $G[V_3,V_5]$.

$G$ is $(4,1,1,1,1)$-, $(3,2,1,1,1)$-, or $(2,2,2,1,1)$-colourable. By Lemma~\ref{high-chromatic-number}, we can assume that $G$ is not $(4,1,1,1,1)$-colourable.

If $G$ is $(2,1,1,1,3)$-colourable as in the picture below, which is equivalent to $G$ being $(3,2,1,1,1)$-colourable, then $G[\{v_3,v_4,v_5\}]$ is a triangle; $e(V_i,v_j)\geq 1$ for $i=1,5$ and $j=2,3,4$ or else we can recolour some vertex in $\{v_2,v_3,v_4\}$ and obtain a 4-colouring of $G$, which is impossible; $e(v_1,V_5)\geq 1$ and $e(v'_1,V_5)\geq 1$, and hence $e(V_1,V_5)\geq 2$, or we can recolour some vertices and get a (4,1,1,1,1)-colouring.

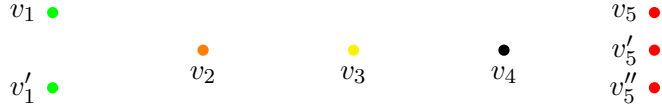
\begin{center}
	\begin{tikzpicture}
		
		\node[circle, fill=red, inner sep=1.5pt, label=left:$v_5$] (A) at (8,0) {};
		\node[circle, fill=red, inner sep=1.5pt, label=left:$v'_5$] (B) at (8,-0.5) {};
		\node[circle, fill=red, inner sep=1.5pt, label=left:$v''_5$] (C) at (8,-1) {};
		\node[circle, fill=green, inner sep=1.5pt, label=left:$v_1$] (D) at (0,0) {};
		\node[circle, fill=green, inner sep=1.5pt, label=left:$v'_1$] (E) at (0,-1) {};
		\node[circle, fill=orange, inner sep=1.5pt, label=below:$v_2$] (F) at (2,-0.5) {};
		\node[circle, fill=yellow, inner sep=1.5pt, label=below:$v_3$] (G) at (4,-0.5) {};
		\node[circle, fill=black, inner sep=1.5pt, label=below:$v_4$] (H) at (6,-0.5) {};
		
	\end{tikzpicture}
\end{center} 

We first prove the following fact, which will be used multiple times below: if there are at least two vertices among $v_2,v_3,v_4$ that have only one neighbour in $V_5$, then $G$ is 1-11-representable. W.l.o.g., we assume $e(v_3,V_5)=e(v_4,V_5)=1$ and $v_3v_5\in E(G)$. By Lemma~\ref{symmetricvertices}, $v_5$ is adjacent to $v_2$, $v_3$, and $v_4$. Now, by adding all edges in $G[V_1,v_3]$ and $G[V_1,v_4]$, we add at most two edges, which can only result from applying a star or matching operation. By Lemma~\ref{type2-lem} or Lemma~\ref{adding-matching}, we claim that there is no shortcut now, so the original graph $G$ is 1-11-representable. Indeed, possible
(12345)-, (1345)-, (1235)-, (1245)-, or (2345)-shortcuts must end with edge $v_3v_5$ or $v_4v_5$, but $G[V_1,V_3]$, $G[V_1,V_4]$, and $G[V_2,V_4]$ are complete bipartite and $v_2v_5,v_3v_5\in E(G)$. Moreover, (1234)-shortcuts do not exist because $G[V_1,V_3]$ and $G[V_2,V_4]$ are complete bipartite. Therefore, the orientation is indeed semi-transitive, and $G$ is indeed 1-11-representable. Hence, in the rest of the proof, we may assume that there are at least two vertices among $v_2,v_3,v_4$ that have more than one neighbours in $V_5$.

Now, let us discuss the different cases based on the possible values of $e(V_1, v_i)$, where $i\in\{2,3,4\}$. Consider the multiset $\{e(V_1, v_i)|i\in\{2,3,4\}\}$. Let $\kappa$ be the number of occurrences of 1 in this multiset. We consider four cases.

\begin{itemize}
\item[i)] $\kappa=0$, which means $e(V_1,v_i)=2$ for $i=2,3,4$. We can assume that $e(v_2,V_5)=e(V_2,V_5)\geq 2$ and $e(v_3,V_5)=e(V_3,V_5)\geq 2$ as stated before. By adding at most two edges we can make $G[V_2,V_5]$ and $G[V_3,V_5]$ complete bipartite. Note that $G[V_1,V_3]$, $G[V_1,V_4]$, and $G[V_2,V_4]$ are already complete bipartite, so there is no shortcut in the above orientation and $G$ is 1-11-representable.

\item[ii)] $\kappa=1$. By recolouring $v_3,v_4,v_5$, if necessary, we can assume that $e(V_1,v_2)=1$, $e(V_1,v_3)=2$, $e(V_1,v_4)=2$ and $v_1v_2\in E(G)$. By symmetry, we can assume that $e(v_2,V_5)\geq e(v_1,V_5)$ (if not, swap $v_1$ and $v_2$). We can add to $G$, by a matching or star operation, edge $v_2v_5$ (resp., $v_2v'_5$, $v_2v''_5$) if $v_1v_5$ (resp., $v_1v'_5$, $v_1v''_5$) is an edge in $E$, and edges
$\{v_3v_5,v_3v'_5,v_3v''_5\}$. In fact, we need to add at most two edges because $e(v_2,V_5)\geq e(v_1,V_5)$ and $e(v_3,V_5)\geq 2$. We claim that then there is no shortcut in the above orientation:
(12--)-shortcuts must start with $v_1v_2$, but $N_{V_5}(v_1)\subseteq N_{V_5}(v_2)$ and
$G[v_1,V_3]$, $G[v_1,V_4]$, $G[V_2,V_4]$, $G[V_2,V_5]$, and $G[V_3,V_5]$ are complete bipartite so there is no such shortcut; (1345), (2345)-shortcuts do not exist because $G[V_1,V_4]$, $G[V_3,V_5]$, and $G[V_2,V_4]$ are complete bipartite.

\item[iii)] $\kappa=2$. By recolouring $v_3,v_4,v_5$, if necessary,  we can assume that $e(V_1,v_2)=1,e(V_1,v_3)=1,e(V_1,v_4)=2$. Similarly to the above, we assume that $v_1v_2\in E(G)$ then $v_1v_3,v_1v_4\in E(G)$. By symmetry, we assume that $e(v_2,V_5)\geq e(v_1,V_5)$. We can add at most one edge to ensure $N_{V_5}(v_1)\subseteq N_{V_5}(v_2)$, and then add at most one additional edge to ensure that $G[V_3,V_5]$ is complete bipartite. Now there is no shortcut.

\item[iv)] $\kappa=3$, which means that $e(V_1,v_i)=1$ for $i=2,3,4$. We assume $v_1v_2\in E(G)$ then $v_1v_3,v_1v_4\in E(G)$. By symmetry, we assume $e(v_2,V_5)\geq e(v_1,V_5)$. Using the same method as in Case iii), we get a semi-transitive orientation by adding at most 2 edges, which means $G$ is 1-11-representable.
\end{itemize}
If $G$ is (2,1,1,2,2)-colourable as in the picture below, we can assume $G[V_1,V_4],G[V_1,V_5],G[V_4,V_5]$ all have a perfect matching or we can recolour some vertices and get a (3,2,2,1,1)-colouring. 

\begin{center}
	\begin{tikzpicture}
		
		\node[circle, fill=red, inner sep=1.5pt, label=left:$v_1$] (A) at (0,0) {};
		\node[circle, fill=red, inner sep=1.5pt, label=left:$v'_1$] (B) at (0,-1) {};
		\node[circle, fill=green, inner sep=1.5pt, label=left:$v_2$] (C) at (2,-0.5) {};
		\node[circle, fill=orange, inner sep=1.5pt, label=left:$v_3$] (D) at (4,-0.5) {};
		\node[circle, fill=yellow, inner sep=1.5pt, label=left:$v_4$] (E) at (6,0) {};
		\node[circle, fill=yellow, inner sep=1.5pt, label=left:$v'_4$] (F) at (6,-1) {};
		\node[circle, fill=black, inner sep=1.5pt, label=left:$v_5$] (G) at (8,0) {};
		\node[circle, fill=black, inner sep=1.5pt, label=left:$v'_5$] (H) at (8,-1) {};
		
	\end{tikzpicture}
\end{center} 

For $i\in\{2,3\}$, let $f(i)=(e(V_1,v_i),e(v_i,V_4),e(v_i,V_5))$ be a triple, where each component is in $\{1,2\}$. For $j\in\{1,4,5\}$, denote $f(i)(j)=e(v_i,V_j)$. By symmetry, we only need to consider cases when: i) for any coordinate $j\in\{1,2,3\}$, there is $i\in\{2,3\}$ such that $f(i)(j)=2$, and for any $i\in\{2,3\}$, $f(i)\neq (1,1,1)$;
ii) some $f(i)=(2,2,1)$; iii) at least two of coordinates of $f(2)$ and $f(3)$ is $2$; iv) $f(2)=(1,1,2)$, $f(3)=(2,1,1)$; v) $f(2)=(1,1,1)$, $f(3)=(2,2,2)$.

\begin{itemize}
	\item[i)] If $e(v_2,V_4)=e(v_2,V_5)=e(v_3,V_1)=2$, then we can add a matching to make $G[V_1,V_4]$ and $G[V_3, V_5]$ complete bipartite. But $G[V_1,V_3]$, $G[V_2,V_4]$, and $G[V_2,V_5]$ are already complete bipartite, so there is no shortcut in the above orientation and the original graph is 1-11-representable.
	
	\item[ii)] If $e(V_1,v_2)=1$, $e(v_2,V_4)=e(v_3,V_5)=2$, by Lemma~\ref{symmetricvertices}, we can assume $v_1v_2,v_1v_3\in E(G)$. After adding all edges in $G[V_1,V_4],G[V_3,V_5]$ we add at most a matching or we can recolour some vertices and get a (3,2,1,1,1)-colouring or a 4-colouring. We claim that then there is no shortcut in the above orientation and the original is 1-11-representable:
	$G[V_1,V_4]$, $G[V_3,V_5]$, $G[V_2,V_4]$, and $G[V_2,V_5]$ are complete bipartite, and shortcuts $G[V_1,V_3]$ are (123--)-shortcuts, which starts from $v_1v_2$, $v_2v_3$, but $v_1v_3\in E(G)$ so no such shortcut exists.
	
	\item[iii)] If $e(V_1,v_2)=e(V_1,v_3)=e(v_2,V_4)=e(v_3,V_4)=1$, by Lemma~\ref{symmetricvertices}, we can assume that $v_1v_2,v_1v_3,v_2v_4,v_3v_4\in E(G)$.
	Now we see that $v_1v_4\in E(G)$, or we can recolour $v_1$ and $v_4$ green, recolour $v_2$ red, recolour $v_3$ yellow and get a 4-colouring, contradiction.
	By adding at most two edges we can make $G[V_2,V_5]$ and $G[V_3,V_5]$ complete bipartite. We claim that then there is no shortcut in the above orientation and the original is 1-11-representable. Indeed, 
	(1--)-shortcuts must start from $v_1v_2\rightarrow v_2v_3\rightarrow v_3v_4$ or $v_1v_2\rightarrow v_2v_4$, but $v_1v_3,v_1v_4,v_2v_4\in G$ and $G[V_2,V_5]$ and $G[V_3,V_5]$ are complete bipartite, so no such shortcut exists.
	
	\item[iv)] If $e(V_1,v_2)=e(v_2,V_4)=e(v_3,V_4)=e(v_3,V_5)=1$ and $e(v_2,V_5)=e(V_1,v_3)=2$, then by symmetry and Lemma~\ref{symmetricvertices} we can assume that $v_1v_2,v_2v_4,v_3v_4,v_3v_5\in E(G)$. By adding at most a matching (or we can get a (3,2,1,1,1)-colouring or a 4 colouring) we can make $G[V_1,V_4]$ and $G[V_3,V_5]$ complete bipartite. We claim that then there is no shortcut in the above orientation and the original graph is 1-11-representable:
	(12--)-shortcuts must start from $v_1v_2\rightarrow v_2v_3\rightarrow v_3v_4$ or $v_1v_2\rightarrow v_2v_4$, but $v_1v_3,v_2v_4\in E(G)$ and $G[V_1,V_4]$, $G[V_2,V_5]$, and $G[V_3,V_5]$ are complete bipartite, so no such shortcut exists; (1345)-shortcuts do not exist because $G[V_1,V_4]$ and $G[V_3,V_5]$ are complete; (2345)-shortcuts must start from $v_2v_3\rightarrow v_3v_4$, and so it's easy to see that no such shortcuts exists.
	
	\item[v)] If $e(v_2,V_i)=1,e(v_3,V_i)=2$ for $i\in\{1,4,5\}$, by symmetry we can assume $v_1v_2,v_2v_4,v_2v_5\in E(G)$. Now $V_1,V_4,V_5$ are symmetric. By switching $v_i,v_2$ for $i\in\{1,5,7\}$, we see that $e(v_1,V_4)=e(v_1,V_5)=e(V_1,v_4)=e(v_4,V_5)=e(V_1,v_5)=e(V_4,v_5)=1$. We claim that $G[\{v_1,v_4,v_5\}]$ is not the empty graph: or else recolour $v_1,v_4,v_5$ red and recolour $v'_1,v_2$ green, we get a (3,2,1,1,1)-colouring. Then we can assume $v_1v_5\in E(G)$. Then adding edge $v_1v_4$ (resp., $v_1v'_4$) if $v_4v_5\in E(G)$  (resp., $v'_4v_5\in E(G)$) and $v_2v'_4$, we add at most two edges and
	claim that there is no shortcut in the above orientation and the original graph is 1-11-representable:
	$G[V_1,V_3]$, $G[V_3,V_5]$, and $G[V_2,V_4]$ are complete bipartite;
	shortcuts concerning $G[V_1,V_4],G[V_2,V_5]$ must start from $v_1$ and end at $v_4$, but $N_{V_5}(v_1)= N_{V_5}(v_2)$ and $N_{V_4}(v_1)\supseteq N_{V_4}(v_5)$. So there is no shortcut. 
\end{itemize}

\noindent
{\bf Case 6: $\chi(G)=4$.} The only shortcuts in this graph are only (1234)-shortcuts and possible missing edges appear only in $G[V_1,V_3],G[V_2,V_4]$.

If $G$ is $(5,1,1,1)$-colourable then, by Lemma~\ref{high-chromatic-number}, $G$ is 1-11-representable.

If $G$ is $(4,2,1,1)$-colourable as in the picture below, we can assume $e(V_2,v_3)\leq e(V_2,v_4)$ by symmetry. If $e(V_2,v_4)=2$, we can add all edges in $G[V_1,v_3]$ by adding at most a star subgraph. Then $G[V_1,V_3]$ and $G[V_2,V_4]$ are complete bipartite and there is no shortcut in the above orientation. So the original graph is 1-11-representable.
If $e(V_2,v_3)=e(V_2,v_4)=1$, by Lemma~\ref{symmetricvertices}, we can assume $v_2v_3,v_2v_4\in E(G)$. Then still we add all edges in $G[V_1,v_3]$ by adding at most a star subgraph and there is no shortcut in the above orientation. Indeed, a (1234)-shortcut must end with $v_2v_3$ and $v_3v_4$, but $G[V_1,V_3]$ and $G[v_2,V_4]$ are complete bipartite, which is a contradiction. So, the original graph is 1-11-representable.

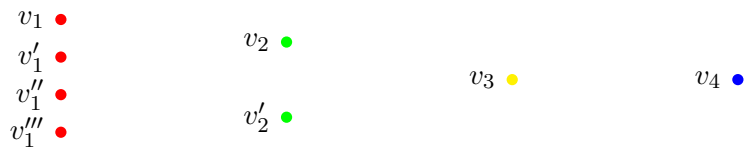
\begin{center}
    \begin{tikzpicture}

    \node[circle, fill=red, inner sep=1.5pt, label=left:$v_1$] (A) at (0,0.3) {};
    \node[circle, fill=red, inner sep=1.5pt, label=left:$v'_1$] (B) at (0,-0.2) {};
    \node[circle, fill=red, inner sep=1.5pt, label=left:$v''_1$] (C) at (0,-0.7) {};
    \node[circle, fill=red, inner sep=1.5pt, label=left:$v'''_1$] (D) at (0,-1.2) {};
    \node[circle, fill=green, inner sep=1.5pt, label=left:$v_2$] (E) at (3,0) {};
    \node[circle, fill=green, inner sep=1.5pt, label=left:$v'_2$] (F) at (3,-1) {};
    \node[circle, fill=yellow, inner sep=1.5pt, label=left:$v_3$] (G) at (6,-0.5) {};
    \node[circle, fill=blue, inner sep=1.5pt, label=left:$v_4$] (H) at (9,-0.5) {};

    \end{tikzpicture}
\end{center} 

If $G$ is $(3,3,1,1)$-colourable as in the picture below, we see that $e(V_i,V_j)\geq 1$ for $i\in\{1,2\},j\in\{3,4\}$. If $e(V_2,v_3)=1$, by Lemma~\ref{symmetricvertices} we can assume $v_2v_3,v_2v_4\in E(G)$.
We can add all edges in $G[V_1,v_3]$ by adding at most a star subgraph and there is no shortcut in the above orientation. Indeed, a (1234)-shortcut must end with $v_2v_3$ and $v_3v_4$, but $G[V_1,V_3]$ and $G[v_2,V_4]$ are complete bipartite, which is a contradiction. So the original graph is 1-11-representable. Now by symmetry we can assume $e(V_i,V_j)\geq 2$ for $i\in\{1,2\}$, $j\in\{3,4\}$. Then, by adding at most two edges we can make $G[V_1,V_3]$ and $G[V_2,V_4]$ complete bipartite and there is no shortcut in the above orientation. So the original graph is 1-11-representable.

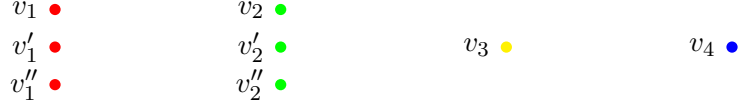
\begin{center}
    \begin{tikzpicture}

    \node[circle, fill=red, inner sep=1.5pt, label=left:$v_1$] (A) at (0,0) {};
    \node[circle, fill=red, inner sep=1.5pt, label=left:$v'_1$] (B) at (0,-0.5) {};
    \node[circle, fill=red, inner sep=1.5pt, label=left:$v''_1$] (C) at (0,-1) {};
    \node[circle, fill=green, inner sep=1.5pt, label=left:$v_2$] (D) at (3,0) {};
    \node[circle, fill=green, inner sep=1.5pt, label=left:$v'_2$] (E) at (3,-0.5) {};
    \node[circle, fill=green, inner sep=1.5pt, label=left:$v''_2$] (F) at (3,-1) {};
    \node[circle, fill=yellow, inner sep=1.5pt, label=left:$v_3$] (G) at (6,-0.5) {};
    \node[circle, fill=blue, inner sep=1.5pt, label=left:$v_4$] (H) at (9,-0.5) {};

    \end{tikzpicture}
\end{center} 

If $G$ is $(3,2,1,2)$-colourable as in the picture below, 
we can assume that $G$ is not $(3,3,1,1)$-colourable, so we can add a matching into $G[V_2,V_4]$ to make it a complete bipartite graph. If $e(V_1,v_3)\geq 2$, 
then we can add a matching into $G$ to make $G[V_1, V_3]$ and $G[V_2,V_4]$ complete bipartite graphs and then there is no shortcut under the above orientation, and so the original graph is 1-11-representable.

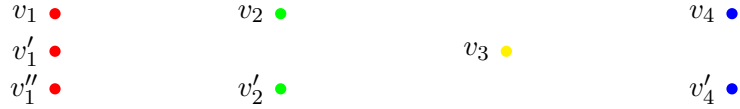
\begin{center}
    \begin{tikzpicture}

    \node[circle, fill=red, inner sep=1.5pt, label=left:$v_1$] (A) at (0,0) {};
    \node[circle, fill=red, inner sep=1.5pt, label=left:$v'_1$] (B) at (0,-0.5) {};
    \node[circle, fill=red, inner sep=1.5pt, label=left:$v''_1$] (C) at (0,-1) {};
    \node[circle, fill=green, inner sep=1.5pt, label=left:$v_2$] (D) at (3,0) {};
    \node[circle, fill=green, inner sep=1.5pt, label=left:$v'_2$] (E) at (3,-1) {};
    \node[circle, fill=yellow, inner sep=1.5pt, label=left:$v_3$] (F) at (6,-0.5) {};
    \node[circle, fill=blue, inner sep=1.5pt, label=left:$v_4$] (G) at (9,0) {};
    \node[circle, fill=blue, inner sep=1.5pt, label=left:$v'_4$] (H) at (9,-1) {};

    \end{tikzpicture}
\end{center} 

If $e(V_1,v_3)=1$, then by swapping $V_2$ and $V_3$ we can assume $E(V_1,V_2)=\{v_1v_2\}$ as in the picture below. In this case, we can add a matching into $G$ to make $G[v_1,V_3]$ and $G[V_2,V_4]$ complete bipartite graphs and then there is no shortcut under the above orientation: (1234)-shortcuts must start with $v_1v_4$, but $G[v_1,V_3]$ and $G[V_2,V_4]$ are complete bipartite. So the original graph is 1-11-representable.

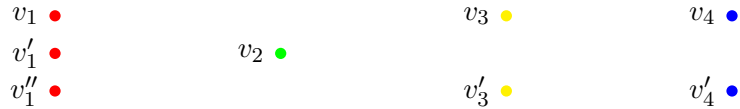
\begin{center}
    \begin{tikzpicture}

    \node[circle, fill=red, inner sep=1.5pt, label=left:$v_1$] (A) at (0,0) {};
    \node[circle, fill=red, inner sep=1.5pt, label=left:$v'_1$] (B) at (0,-0.5) {};
    \node[circle, fill=red, inner sep=1.5pt, label=left:$v''_1$] (C) at (0,-1) {};
    \node[circle, fill=green, inner sep=1.5pt, label=left:$v_2$] (D) at (3,-0.5) {};
    \node[circle, fill=yellow, inner sep=1.5pt, label=left:$v_3$] (E) at (6,0) {};
    \node[circle, fill=yellow, inner sep=1.5pt, label=left:$v'_3$] (F) at (6,-1) {};
    \node[circle, fill=blue, inner sep=1.5pt, label=left:$v_4$] (G) at (9,0) {};
    \node[circle, fill=blue, inner sep=1.5pt, label=left:$v'_4$] (H) at (9,-1) {};

    \end{tikzpicture}
\end{center} 

If $G$ is $(2,2,2,2)$-colourable as in the picture below, we can assume that $G$ is not $(3,2,2,1)$-colourable. Then we can add a matching into $G$ to make $G[V_1, V_3]$ and $G[V_2,V_4]$ complete bipartite graphs or we can find a vertex $v_i$ such that there is a component not containing it and there is no edge between this vertex and this component. By adding this vertex to the component, we will get a $(3,2,2,1)$-colouring, contradiction. Thus there is no shortcut under the above orientation, and so $G$ is 1-11-representable.

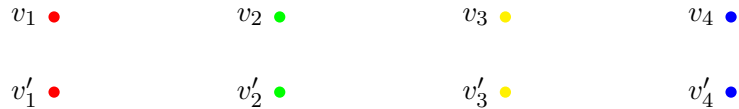
\begin{center}
    \begin{tikzpicture}

    \node[circle, fill=red, inner sep=1.5pt, label=left:$v_1$] (A) at (0,0) {};
    \node[circle, fill=red, inner sep=1.5pt, label=left:$v'_1$] (B) at (0,-1) {};
    \node[circle, fill=green, inner sep=1.5pt, label=left:$v_2$] (C) at (3,0) {};
    \node[circle, fill=green, inner sep=1.5pt, label=left:$v'_2$] (D) at (3,-1) {};
    \node[circle, fill=yellow, inner sep=1.5pt, label=left:$v_3$] (E) at (6,0) {};
    \node[circle, fill=yellow, inner sep=1.5pt, label=left:$v'_3$] (F) at (6,-1) {};
    \node[circle, fill=blue, inner sep=1.5pt, label=left:$v_4$] (G) at (9,0) {};
    \node[circle, fill=blue, inner sep=1.5pt, label=left:$v'_4$] (H) at (9,-1) {};

    \end{tikzpicture}
\end{center} 

All cases have been considered; thus, the theorem is proved.
\end{proof}

\section{Multi-1-11-representation number of a graph}\label{multi-1-11-repr-sec}

In this section, we generalize and extend the notion of the {\em multi-word-representation number} of a graph, introduced in \cite{KM} by Kenkireth and Malhotra. The key idea involves using multiple words over the same alphabet to represent different graphs, and declaring that the union of these word-representants represents the union of graphs. 

\begin{definition}\label{multi-represent} Suppose that the graphs $G_1, G_2, \ldots, G_m$ share the same vertex set, i.e., $V(G_1)=\cdots=V(G_m)=V$ and that the graphs $G$ and $G'$ satisfy the following: 
\begin{itemize}
\item $V(G)=V$ and $E(G)=\cup_{1\leq i\leq m}E(G_i)$;
\item $V(G')=V$, $E(G')=\cup_{1\leq i\leq m}E(G_i)$, and $E(G_i)\cap E(G_j)=\emptyset$ for $1\leq i<j\leq m$.
\end{itemize}
Further assume that each $G_i$, $1\leq i\leq m$, is $k$-11-representable, and that $m$ is {\em minimal possible} value for $G$ and $G'$. Then, we define the  (resp., {\em strict}) {\em multi-$k$-$11$-representation number} of $G$ (resp., $G'$) to be $m$.
\end{definition}

Since each graph is $k$-11-representable for $k\geq 2$ (see \cite{CKKKP2019}), the (strict) multi-$k$-11-representation number of such a graph is 1. Hence, Definition~\ref{multi-represent} is meaningful only in the case $k\in\{0,1\}$. In particular, unless it is proven that all graphs are 1-11-representable (which might not be true!), establishing the (strict) multi-$k$-11-representation number for graphs, or classes of graphs, remains an interesting and challenging problem. Furthermore, note that the multi-1-11-representation number is clearly at most equal to the strict multi-1-11-representation number.

Using the approach in \cite{KM} and applying our results from Section~\ref{graph-8-vertices}, we can prove the following theorem.

\begin{theorem}\label{thm-24} All graphs on at most $24$ vertices have a strict multi-$1$-$11$-representation number of at most $2$. \end{theorem}

\begin{proof} Suppose $G$ is a graph on 24 vertices (for smaller graphs, the statement will follow by the hereditary nature of 1-11-representation). Consider an arbitrary partition of $V(G)$ into three disjoint subsets $V_1$, $V_2$, and $V_3$, each containing 8 vertices.  By Theorem~\ref{thm-main}, the graph $G[V_i]$, $i\in\{1,2,3\}$, is 1-11-representable. Furthermore, by Lemma~\ref{type1-lem}(a), the graph $G':=G[V_1]\cup G[V_2]\cup G[V_3]$ can be 1-11-represented by a word $w_1$.  

Next, removing all edges within  $G[V_1]$, $G[V_2]$, and $G[V_3]$ from $G$, we obtain a 3-colourable graph $G''$, which is word-representable and therefore 1-11-representable \cite{CKKKP2019}. Thus, we can find a word $w_2$ that 1-11-represents $G''$.

Since $V(G)=V(G')=V(G'')$, $E(G)=E(G')\cup E(G'')$, and $E(G')\cap E(G'')=\emptyset$, the strict 1-11-representation number of $G$ is at most 2. \end{proof}

\section{Concluding remarks}\label{concl-sec}

We conclude this paper with several open directions for future research:

\begin{itemize}
\item Are all 4-colourable graphs 1-11-representable? If this is the case, the result of Theorem~\ref{thm-24} could be immediately improved by replacing 24 with 32. Indeed, in the proof of that theorem, we could partition the vertex set of $G$ into four sets, each containing 8 vertices, and use the 1-11-representability of the 4-colourable $G''$.

\item If proving or disproving that all 4-colourable graphs are 1-11-representable proves too challenging, one could instead address the same question for all planar graphs, a subclass of 4-colourable graphs.

\item Assuming that proving or disproving that any graph is 1-11-representable remains infeasible with existing tools, one could focus on proving or disproving that the (strict) multi-1-11-representation number of any graph is at most 2. This question is likely easier, at least for various classes of graphs, than proving 1-11-representability. It should also be more tractable than resolving the open problem of whether the multi-word-representation number of any graph is at most 2 (since graphs can be modified by adding edges).

\item The notions of strict and non-strict multi-$k$-11-representation numbers are equivalent for $k \geq 2$. What can be said about $k \in {0,1}$? Is it possible to construct any counterexamples in this case?

\item Definition~\ref{multi-represent} can, in fact, be refined to the $\ell$-multi-$k$-11-representation number, where any edge can belong to at most $\ell$ subgraphs. In this framework, the strict multi-$k$-11-representation corresponds to the case $\ell=1$. Could such a refinement lead to interesting results for $k=0$ (word-representation) or $k=1$ (assuming not all graphs are 1-11-representable)?

\item Finally, Herman Chen's experiments~\cite{Chen} suggest that the 1-11-representability of graphs on 8 vertices can be established by adding at most one new edge (each of the 929 non-word-representable graphs can be converted into a word-representable graph by adding a single edge). However, our arguments in Section~\ref{graph-8-vertices} often rely on adding more than one edge. Is it possible to prove (not computationally) that adding at most one edge is sufficient? Such a proof could lead to useful techniques, for example, to establish the 1-11-representability of all graphs with 9 vertices (if they are indeed 1-11-representable).
\end{itemize}

\section*{Acknowledgments} The authors are grateful to Brian Ritchie and Hehui Wu for useful discussions related to the paper.

\end{document}